\documentclass[11pt,twoside, final]{amsart}
\usepackage{amssymb}
\usepackage{amsmath,amsthm,amscd,amsfonts,amssymb,enumerate}
\usepackage{graphicx}		
\usepackage{color}
\usepackage[colorlinks]{hyperref}

\newtheorem{theorem}{Theorem}[section]

\newtheorem{corollary}[theorem]{Corollary}
\theoremstyle{definition}

\theoremstyle{remark}

\numberwithin{equation}{section}

\begin{document}

\title[Generalized beta-type integral operators]{Generalized beta-type integral operators}

\author[Musharraf Ali]{Musharraf Ali}
\address[Musharraf Ali]{Department of Mathematics, G.F. College, Shahjahanpur-242001, India}
\email{drmusharrafali@gmail.com}

\author[Mohd Ghayasuddin]{Mohd Ghayasuddin}
\address[Mohd Ghayasuddin]{Department of Mathematics, Faculty of Science, Integral University, Lucknow-226026, India}
\email{ghayas.maths@gmail.com}

\author[R. B. Paris]{R. B. Paris $^*$}
\address[R. B. Paris]{Division of Computing and Mathematics, Abertay University, Dundee DD1 1HG, UK}
\email{r.paris@abertay.ac.uk}


  \thanks{$^*$Corresponding author}

\vspace{.50cm}
\parindent=5mm
 \maketitle
%

\begin{abstract}
This research note deals with the evaluation of some generalized beta-type integral operators involving the multi-index Mittag-Leffler function $E_{\epsilon_{i}),(\omega_{i})}(z)$. Further, we derive a new family of beta-type integrals involving the product of a multi-index Mittag-Leffler function and a generating function of two variables. Some concluding remarks regarding our present investigation are briefly discussed in the last section.\\

\end{abstract}

\vspace{.50cm}
\parindent=8mm
\section{\bf Introduction}
Several basic functions in the applied sciences are defined via improper integrals, which are predominantly called special functions. Such functions play a remarkable role in various diverse fields of engineering and sciences. Therefore, many researchers have presented several extensions and associated properties of such type of functions.

In recent years, a number of authors have established a list of integral formulas associated with different kinds of special functions (see, for details,  \cite{a7}--\cite{saima}, \cite{a1}--\cite{a6} and the references cited therein). In order to extend the above-mentioned literature, we present in this paper a new class of beta-type integral operators associated with the multi-index Mittag-Leffler function and a generating function of two variables.

Throughout, let $\mathbb{N}$, $\mathbb{Z}$, $\mathbb{R}$, $\mathbb{R^{+}}$ and $\mathbb{C}$ be the sets of natural numbers, integers, real numbers, positive real numbers and complex numbers, respectively.

The generalized Wright hypergeometric function ${}_m{\Psi}_{n}$ is defined by (see \cite{s21}, see also \cite{g1} )
\begin{equation}\label{1.1} {}_m{\Psi}_{n} \left[\begin{array}{cccc}
(\lambda_{1}, G_{1}), & \cdots, & (\lambda_{m}, G_{m}); \\
~ & ~ & ~ \\
(\mu_{1}, H_{1}), & \cdots, & (\mu_{n}, H_{n}); \\
\end{array}x\right]=\sum_{k=0}^{\infty}\frac{\prod\limits_{j=1}^{m}\Gamma(\lambda_{j}+G_{j}k)}{\prod\limits_{j=1}^{n}\Gamma(\mu_{j}+H_{j}k)}\frac{x^{k}}{k!},
\end{equation}
where the coefficients $G_{j}\in \mathbb{R^{+}}~(j=1,\ldots ,m)$ and
$H_{j}\in \mathbb{R^{+}}~(j=1,\ldots ,n)$ are such that
\begin{equation}\label{1.2} 1+\sum_{j=1}^{n}H_{j}-\sum_{j=1}^{m}G_{j}\geq{0}.\end{equation}

The Mittag-Leffler function $E_{\lambda}(z)$ was introduced by the Swedish mathematician Gosta Mittag-Leffler \cite{g3} and is defined as follows:
\begin{equation}\label{1.3} E_{\lambda}(z)=\sum_{k=0}^{\infty}\frac{z^{k}}{\Gamma(1+\lambda
k)},
\end{equation}
where $z\in \mathbb{C}$ and $\lambda\geq0$.
The Mittag-Leffler function is a direct generalization of the exponential function to which it reduces when $\lambda=1$. Its importance has been
realized during the last two decades due to its involvement in problems of physics,
chemistry, biology, engineering and applied sciences. The Mittag-Leffler function
occurs naturally as the solution of certain fractional-order differential and integral equations.

Wiman \cite{g4} introduced the following extension of $E_{\lambda}(z)$:
\begin{equation}\label{1.4} E_{\lambda,\mu}(z)=\sum_{k=0}^{\infty}\frac{z^{k}}{\Gamma(\mu+\lambda
k)},\end{equation}
which is known as the Wiman function. The properties of the Wiman function $E_{\lambda,\mu}(z)$
and the Mittag-leffler function $E_{\lambda}(z)$ are very similar. These functions play a very important role in the solution of
differential equations of fractional order.

Subsequently, Kiryakova \cite{g2} introduced the following new class of multi-index Mittag-Leffler functions:
\begin{equation}\label{1.5} E_{(\frac{1}{\epsilon_{i}}),(\omega_{i})}(z)=\sum_{k=0}^{\infty}\frac{z^{k}}{\Gamma({\omega}_{1}+\epsilon_{1}k)\cdots\Gamma(\omega_{l}+\epsilon_{l}k)},
\end{equation}
where $l>1$ is an integer, ${\epsilon}_{1},\ldots ,{\epsilon}_{l}>0$ and
${\omega}_{1},\ldots ,{\omega_{l}}$ are arbitrary real numbers. Some special cases of \eqref{1.5} for $l=2$ are given as follows (see \cite{g2}, \cite{g5}):\\

\noindent $\bullet$ If 
$\epsilon_{1}=\lambda$, $\epsilon_{2}=0$, and $\omega_{1}=\mu$, $\omega_{2}=1$ then \eqref{1.5} reduces to \eqref{1.4} (and to \eqref{1.3} when $\mu=1$).\\



\noindent$\bullet$ If $\epsilon_{1}=\epsilon_{2}=1$ and $z$ is replaced by $z^2/4$ then
we have  \cite{s21}
\begin{equation}\label{1.6} 
\omega_1=1+\nu, \omega_2=1:\qquad E_{(1,1),(1+\nu,1)}(-z^{2}/4)=\left(\frac{1}{2}z\right)^{-\nu}J_{\nu}(z),
\end{equation}


\begin{equation}\label{1.7} 
\omega_1=\frac{3+\mu-\nu}{2}, \omega_2=\frac{3+\mu+\nu}{2}:\quad  E_{(1,1),(\frac{3-\nu+\mu}{2},\frac{3+\nu+\mu}{2})}(-z^{2}/4)=\frac{4}{z^{\mu+1}}~S_{\mu,\nu}(z),
\end{equation}

and
\begin{equation}\label{1.8} 
\omega_1=\frac{3}{2}, \omega_2=\frac{3}{2}+\nu:\qquad E_{(1,1),(\frac{3}{2},\frac{3+2\nu}{2})}(-z^{2}/4)=\frac{4}{z^{\mu+1}}~H_{\nu}(z),
\end{equation}
where $J_\nu(z)$, $S_{\mu,\nu}(z)$ and $H_{\nu}(z)$ are  respectively the Bessel function of the first kind, the Struve function and the Lommel function.

\section{\bf Evaluation of Beta-type integrals}
In this section, we present some theorems on the evaluation of the beta-type integral
\begin{equation}\label{2.1} \mathbb{J}_{a_{1},a_{2},\epsilon_{i},\omega_{i}}^{\eta_{1}, \eta_{2}, \eta_{3}}[h_{1}(u),h_{2}(u);q]\end{equation}
\begin{equation*} =\frac{1}{B(\eta_{1}, \eta_{2})}\int_{a_{1}}^{a_{2}} (u-a_{1})^{\eta_{1}-1}(a_{2}-u)^{\eta_{2}-1}[h_{1}(u)]^{\eta_{3}}E_{(\frac{1}{\epsilon_{i}}),(\omega_{i})}[qh_{2}(u)]du,
\end{equation*}
for some specific functions $h_{1}(u)$, $h_{2}(u)$ and $q\in {\mathbb{C}}$, where $B(x,y)=\Gamma(x) \Gamma(y)/\Gamma(x+y)$ is the Beta function.

\begin{theorem} The following integral formula holds true:

\begin{equation*} \mathbb{J}_{0,1,\epsilon_{i},\omega_{i}}^{\eta_{1}, \eta_{2}, 1}[(1-z_{1}u)^{-\beta_{1}}(1-z_{2}u)^{-\beta_{2}},u(1-u);q]=\frac{1}{B(\eta_{1}, \eta_{2})} \sum_{r,s=0}^{\infty} \frac{(\beta_{1})_{r}(\beta_{2})_{s}{z_{1}}^{r}{z_{2}}^{s}}{r!~s!}\end{equation*}

\begin{equation}\label{2.2} \times {}_3\Psi_{l+1}
\left[\begin{array}{ccc}
  (\eta_{1}+r+s,1), & (\eta_{2},1), & (1,1); \\
  ~ & ~ & ~ \\
  (\omega_{1}, \epsilon_{1}), & \cdots (\omega_{l}, \epsilon_{l}), & (\eta_{1}+\eta_{2}+r+s, 2);
\end{array} q \right],
\end{equation}
where $\Re(\eta_{1})>0,~\Re(\eta_{2})>0$ and $|z_{1}|<1,~|z_{2}| <1$.
\end{theorem}

\begin{proof}
We have (see \cite[p.962(7)]{s11})
\begin{equation}\label{2.3} \int_{0}^{1}u^{\eta_{1}-1} (1-u)^{\eta_{2}-1}(1-z_{1}u)^{-\beta_{1}}(1-z_{2}u)^{-\beta_{2}}du=B(\eta_{1}, \eta_{2})
\end{equation}
\begin{equation*}
\times F_{1}(\eta_{1}, \beta_{1}, \beta_{2}; \eta_{1}+\eta_{2}; z_{1},z_{2}),
\end{equation*}
where $\Re(\eta_{1})>0,~\Re(\eta_{2})>0$, $|z_{1}|<1,~|z_{2}|< 1$ and $F_{1}$ is the Appell function \cite{s21}; see also \cite[p.~413]{s12}.\\

On setting $a_{1}=0,~a_{2}=\eta_{3}=1,~h_{2}(u)=u(1-u)$ and $h_{1}(u)=(1-z_{1}u)^{-\beta_{1}}(1-z_{2}u)^{-\beta_{2}}$ in \eqref{2.1}, expanding the multi-index Mittag-Leffler function $E_{(\epsilon_{i}),(\omega_{i})}$ by its defining series, applying the result \eqref{2.3} on the right side and after some simplification, we obtain our required result \eqref{2.2}.
\end{proof}

\begin{theorem} The following integral formula holds true:

\begin{equation*} \mathbb{J}_{0,1,\epsilon_{i},\omega_{i}}^{\eta_{1}, \eta_{2}, 1}[(1-z_{1}u)^{-\beta_{1}}(1-z_{2}(1-u))^{-\beta_{2}},u(1-u);q]=\frac{1}{B(\eta_{1}, \eta_{2})} \sum_{r,s=0}^{\infty} \frac{(\beta_{1})_{r}(\beta_{2})_{s}{z_{1}}^{r}{z_{2}}^{s}}{r!~s!}\end{equation*}

\begin{equation}\label{2.4} \times {}_3\Psi_{l+1}
\left[\begin{array}{ccc}
  (\eta_{1}+r,1), & (\eta_{2}+s,1), & (1,1); \\
  ~ & ~ & ~ \\
  (\omega_{1}, \epsilon_{1}), & \cdots (\omega_{l}, \epsilon_{l}), & (\eta_{1}+\eta_{2}+r+s, 2);
\end{array} q \right],
\end{equation}
where $\Re(\eta_{1})>0,~\Re(\eta_{2})>0$ and $|z_{1}|<1,~|z_{2}| <1$.
\end{theorem}

\begin{proof}
We have (see \cite[p.279(17)]{s20})
\begin{equation}\label{2.5} \int_{0}^{1}u^{\eta_{1}-1} (1-u)^{\eta_{2}-1}(1-z_{1}u)^{-\beta_{1}}(1-z_{2}(1-u))^{-\beta_{2}}du=B(\eta_{1}, \eta_{2})
\end{equation}
\begin{equation*}
\times F_{3}(\eta_{1},\eta_{2}, \beta_{1}, \beta_{2}; \eta_{1}+\eta_{2}; z_{1},z_{2}),
\end{equation*}
where $\Re(\eta_{1})>0,~\Re(\eta_{2})>0$, $|z_{1}|<1,~|z_{2}|< 1$ and $F_{3}$ is the Appell function \cite{s21}; see also \cite[p.~413]{s12}.\\

On putting $a_{1}=0,~a_{2}=\eta_{3}=1,~h_{2}(u)=u(1-u)$ and $h_{1}(u)=(1-z_{1}u)^{-\beta_{1}}(1-z_{2}(1-u))^{-\beta_{2}}$ in \eqref{2.1}, expanding the multi-index Mittag-Leffler function $E_{(\epsilon_{i}),(\omega_{i})}$ by its defining series, applying the result \eqref{2.5} on the right side and after some simplification, we obtain our required result \eqref{2.4}.
\end{proof}

\begin{theorem} The following integral formula holds true:

\begin{equation*} \mathbb{J}_{a_{1},a_{2},\epsilon_{i},\omega_{i}}^{\eta_{1}, \eta_{2}, \eta_{3}}[xu+y,(u-a_{1})(a_{2}-u);q]=\frac{1}{B(\eta_{1}, \eta_{2})} \sum_{r=0}^{\infty} \frac{(-\eta_{3})_{r}}{r!}\left\{-\frac{(a_{2}-a_{1})x}{a_{1}x+y}\right\}^{r}\end{equation*}

\begin{equation}\label{2.6} \times {}_3\Psi_{l+1}
\left[\begin{array}{ccc}
  (\eta_{1}+r,1), & (\eta_{2},1), & (1,1); \\
  ~ & ~ & ~ \\
  (\omega_{1}, \epsilon_{1}), & \cdots (\omega_{l}, \epsilon_{l}), & (\eta_{1}+\eta_{2}+r, 2);
\end{array} q \right],
\end{equation}
where $\Re(\eta_{1})>0,~\Re(\eta_{2})>0$, $\left|\arg \left( \frac{a_{2}x+y}{a_{1}x+y}\right)\right|< \pi$ and $a_{1}\neq a_{2}$.
\end{theorem}

\begin{proof}
We have (see \cite[p.263]{s15})
\begin{equation}\label{2.7} \int_{a_{1}}^{a_{2}} (u-a_{1})^{\eta_{1}-1} (a_{2}-u)^{\eta_{2}-1}(xu+y)^{\eta_{3}}du=B(\eta_{1}, \eta_{2})
\end{equation}
\begin{equation*}
\times {}_2F_{1}\left(-\eta_{3},\eta_{1}; \eta_{1}+\eta_{2}; -\frac{(a_{2}-a_{1})x}{a_{1}x+y}\right),
\end{equation*}
where $\Re(\eta_{1})>0,~\Re(\eta_{2})>0$, $\left|\arg \left( \frac{a_{2}x+y}{a_{1}x+y}\right)\right|< \pi$, $a_{1}\neq a_{2}$ and ${}_2F_{1}$ is the Gauss hypergeometric function \cite{s21}; see also \cite[p.~384]{s12}.\\

On setting $h_{1}(u)=xu+y$ and $h_{2}(u)=(u-a_{1})(a_{2}-u)$ in \eqref{2.1}, expanding the multi-index Mittag-Leffler function $E_{(\epsilon_{i}),(\omega_{i})}$ by its defining series, applying the result \eqref{2.7} on the right side and after some simplification, we obtain \eqref{2.6}.
\end{proof}

\begin{theorem} The following integral formula holds true:

\begin{equation*} \mathbb{J}_{a_{1},a_{2},\epsilon_{i},\omega_{i}}^{\eta_{1}, \eta_{2}, -(\eta_{1}+\eta_{2})}\left[(a_{2}-a_{1})+\xi(u-a_{1})+\sigma(a_{2}-u),\frac{(u-a_{1})(a_{2}-u)}{[(a_{2}-a_{1})+\xi(u-a_{1})+\sigma(a_{2}-u)]^{2}};q\right]\end{equation*}

\begin{equation}\label{2.8} =\frac{(\xi+1)^{-\eta_{1}}(\sigma+1)^{-\eta_{2}}}{B(\eta_{1}, \eta_{2})(a_{2}-a_{1})}~{}_3\Psi_{l+1}
\left[\begin{array}{ccc}
  (\eta_{1},1), & (\eta_{2},1), & (1,1); \\
  ~ & ~ & ~ \\
  (\omega_{1}, \frac{1}{\epsilon_{1}}), & \cdots (\omega_{l}, \frac{1}{\epsilon_{l}}), & (\eta_{1}+\eta_{2}, 2);
\end{array} \frac{q}{(\xi+1)(\sigma+1)} \right],
\end{equation}
where $\Re(\eta_{1})>0,~\Re(\eta_{2})>0$, $(a_{2}-a_{1})+\xi(u-a_{1})+\sigma(a_{2}-u) \neq 0$ and $a_{1}\neq a_{2}$.
\end{theorem}

\begin{proof}
We have (see \cite[p.261(3.1)]{s15})
\begin{equation}\label{2.9} \int_{a_{1}}^{a_{2}} \frac{(u-a_{1})^{\eta_{1}-1} (a_{2}-u)^{\eta_{2}-1}~du}{[(a_{2}-a_{1})+\xi(u-a_{1})+\sigma(a_{2}-u)]^{\eta_{1}+\eta_{2}}}= \frac{ B(\eta_{1}, \eta_{2}) (\xi+1)^{-\eta_{1}}(\sigma+1)^{-\eta_{2}}}{(a_{2}-a_{1})},
\end{equation}
where $\Re(\eta_{1})>0,~\Re(\eta_{2})>0$, $(a_{2}-a_{1})+\xi(u-a_{1})+\sigma(a_{2}-u) \neq 0$ and $a_{1}\neq a_{2}$.\\

On taking $h_{1}(u)=(a_{2}-a_{1})+\xi(u-a_{1})+\sigma(a_{2}-u)$, $h_{2}(u)=(u-a_{1})(a_{2}-u)/{h_{1}^{2}}(u)$ and $\eta_{3}=-(\eta_{1}+\eta_{2})$ in \eqref{2.1}, expanding the multi-index Mittag-Leffler function $E_{(\epsilon_{i}),(\omega_{i})}$ by its defining series, applying the result \eqref{2.9} on the right side and after some simplification, we obtain our desired result \eqref{2.8}.
\end{proof}

\section{\bf Beta-type integrals involving a generating function of two variables}

This section deals with some beta-type integrals involving a generating function of two variables. The generating function of two variables $G(u,t)$ is defined as follows \cite{s21}:
\begin{equation}\label{3.1} G(u,t)=\sum_{r=0}^{\infty}{a}_r g_r(u){t}^r,\end{equation}
where each member of the generated set $\left\lbrace g_r(u)\right\rbrace _{r=0}^\infty$ is independent of $t$, and the coefficient set $\left\lbrace a_r\right\rbrace _{r=0}^\infty$ may contain the parameters of the set $\left\lbrace g_r(u)\right\rbrace _{r=0}^\infty$ but is independent of $u$ and $t$.

\begin{theorem}
Let the generating function $G(u,t)$ defined by \eqref{3.1} and be such that $G\left\{u,t{y}^{\mu}(1-y)^{\nu}\right\}$ is uniformly convergent for $y \in(0,1)$, $\mu,~\nu\geq0$ and $\mu+\nu>0$. Then we have
\begin{equation}\label{3.2} \int_{0}^{1} y^{m-1} (1-y)^{n-m-1} G\left\{u,t{y}^{\mu}(1-y)^{\nu}\right\} E_{\epsilon_{i}),(\omega_{i})}[qy(1-y)] dy\end{equation}
\begin{equation*}=\sum_{r=0}^{\infty}{a}_r g_r(u){t}^r~{}_3\Psi_{l+1}
\left[\begin{array}{ccc}
  (m+\mu r,1), & (n-m+\nu r,1), & (1,1); \\
  ~ & ~ & ~ \\
  (\omega_{1}, \epsilon_{1}), & \cdots (\omega_{l}, \epsilon_{l}), & (n+\mu r+\nu r, 2);
\end{array} q \right],
\end{equation*}
where $n>m>0$ and $q\in \mathbb{C}$.
\end{theorem}

\begin{proof}
On using \eqref{3.1} in the left side of \eqref{3.2}, expanding the multi-index Mittag-Leffler function $E_{(\epsilon_{i}),(\omega_{i})}$ as its defining series, applying the definition of beta function and after some simplification, we arrive at \eqref{3.2}.
\end{proof}

\begin{corollary}
On taking $n=2m$ and $\mu=\nu$ in \eqref{3.2}, we have
\begin{equation}\label{3.3} \int_{0}^{1} y^{m-1} (1-y)^{m-1} G\left[u,t\{y(1-y)\}^{\nu}\right] E_{(\epsilon_{i}),(\omega_{i})}[qy(1-y)] dy\end{equation}
\begin{equation*}=\sum_{r=0}^{\infty}{a}_{r} g_{r}(u){t}^r~{}_3\Psi_{l+1}
\left[\begin{array}{ccc}
  (m+\nu r,1), & (m+\nu r,1), & (1,1); \\
  ~ & ~ & ~ \\
  (\omega_{1}, \epsilon_{1}), & \cdots (\omega_{l}, \epsilon_{l}), & (2m+2\nu r, 2);
\end{array} q \right].
\end{equation*}
\end{corollary}
Next, by making use of Theorem 3.1, we derive a few more interesting integrals in the following examples:\\

\noindent {\bf Example 3.1.} Let us consider the generating function of the hypergeometric function \cite[p.~44(8)]{s21}
\begin{equation*} G(u,t) =(1-ut)^{-c}=\sum_{r=0}^{\infty}\frac{(c)_r(ut)^r}{r!}=~{}_1F_{0}(c;-~;ut)\qquad(\left|ut\right| <1).\end{equation*}
By making use of this generating function (with $u=1$) and Theorem 3.1, we obtain
\begin{equation}\label{3.4} \int_{0}^{1} y^{m-1} (1-y)^{n-m-1} [1-t{y}^{\mu}(1-y)^{\nu}]^{-c} E_{(\epsilon_{i}),(\omega_{i})}[qy(1-y)] dy\end{equation}
\begin{equation*}=\sum_{r=0}^{\infty}\frac{{c}_r {t}^r}{r!}~{}_3\Psi_{l+1}
\left[\begin{array}{ccc}
  (m+\mu r,1), & (n-m+\nu r,1), & (1,1); \\
  ~ & ~ & ~ \\
  (\omega_{1}, \epsilon_{1}), & \cdots (\omega_{l}, \epsilon_{l}), & (n+\mu r+\nu r, 2);
\end{array} q \right],
\end{equation*}
where $n>m>0$, $\mu,\nu\geq0$ and $\mu+\nu>0$.\\

\noindent {\bf Example 3.2.} Let us consider the generating function \cite[p.409(2)]{s16}
\begin{equation}\label{3.5} G(u,t)=\sum_{r=0}^{\infty}\frac{(c)_r}{(d)_r}~_1F_1(c;d+r;u)\frac{t^r}{r!}=\Phi_2[c,c;d;u,t]\quad\left( \left|u\right| ,\left| t\right| \right<\infty), \end{equation}
where $\Phi_2$ is Humbert's confluent hypergeometric function \cite{s21}.\\

\noindent Use of this generating function and Theorem 3.1, yields

\begin{equation}\label{3.6} \int_{0}^{1} y^{m-1} (1-y)^{n-m-1}~\Phi_2[c,c;d;u,ty^{\mu}(1-y)^{\nu}] E_{(\epsilon_{i}),(\omega_{i})}[qy(1-y)] dy\end{equation}
\begin{equation*}=\sum_{r=0}^{\infty}\frac{(c)_r {t}^r}{(d)_{r}~r!} {}_1F_1[c;d+r;u]~{}_3\Psi_{l+1}
\left[\begin{array}{ccc}
  (m+\mu r,1), & (n-m+\nu r,1), & (1,1); \\
  ~ & ~ & ~ \\
  (\omega_{1}, \epsilon_{1}), & \cdots (\omega_{l}, \epsilon_{l}), & (n+\mu r+\nu r, 2);
\end{array} q \right],
\end{equation*}
where $n>m>0$, $\mu,\nu\geq0$, $\mu+\nu>0$ and $|u|<\infty$.\\

\noindent {\bf Example 3.3.} Let us consider the generating function \cite[p.276(1)]{s17}
\begin{equation}\label{3.7} G(u,t)=(1-2ut+t^{2})^{-\alpha}=\sum_{r=0}^{\infty}C_{r}^{(\alpha)}(u) t^r,\end{equation}
where the $C_{r}^{(\alpha)}(u)$ are the Gegenbauer, or ultraspherical, polynomials \cite[p.~444]{s12}.\\

\noindent By making use of this generating function (with $u=1$), Theorem 3.1 and fact that $C_{r}^{(\alpha)}(1)=\frac{(2\alpha)_{r}}{r!}$ (see \cite[Eq.(18.5.9)]{s12} ), we find

\begin{equation}\label{3.8} \int_{0}^{1} y^{m-1} (1-y)^{n-m-1}~[1-ty^{\mu}(1-y)^{\nu}]^{-2\alpha} E_{(\epsilon_{i}),(\omega_{i})}[qy(1-y)] dy\end{equation}
\begin{equation*}=\sum_{r=0}^{\infty}\frac{(2 \alpha)_r~{t}^r}{r!}~{}_3\Psi_{l+1}
\left[\begin{array}{ccc}
  (m+\mu r,1), & (n-m+\nu r,1), & (1,1); \\
  ~ & ~ & ~ \\
  (\omega_{1}, \epsilon_{1}), & \cdots (\omega_{l}, \epsilon_{l}), & (n+\mu r+\nu r, 2);
\end{array} q \right],
\end{equation*}
where $n>m>0$, $\mu,\nu\geq0$ and $\mu+\nu>0$.\\

\noindent{\bf Remark.}
On setting $l=2$ with $\epsilon_{1}=\lambda$, $\epsilon_{2}=0$, and $\omega_{1}=\omega_{2}=1$ in all the integrals derived in Sections 2 and 3, we easily recover the results of Jabee {\it et al.} \cite{saima}.

\section{\bf Concluding remarks}
In the present research note, we have evaluated a list of generalized beta-type integral operators involving the multi-index Mittag-Leffler function $E_{\epsilon_{i}),(\omega_{i})}$. We have also pointed out some particular cases of our main results. In this section, we briefly consider the following generalizations of Theorem 2.1 and Theorem 3.1:

\begin{theorem} The following integral formula holds true:

\begin{equation*} \mathbb{J}_{0,1,\epsilon_{i},\omega_{i}}^{\eta_{1}, \eta_{2}, 1}\left[\prod_{i=1}^{n}(1-z_{i}u)^{-\beta_{i}},u(1-u);q\right]=\frac{1}{B(\eta_{1}, \eta_{2})} \sum_{r_{1},\cdots,r_{n}=0}^{\infty} \frac{(\beta_{1})_{r_{1}} \cdots (\beta_{n})_{r_{n}}{z_{1}}^{r_{1}}\cdots {z_{n}}^{r_{n}}}{r_{1}!\cdots r_{n}!}\end{equation*}

\begin{equation}\label{4.1} \times {}_3\Psi_{l+1}
\left[\begin{array}{ccc}
  (\eta_{1}+r_{1}+\cdots +r_{n},1), & (\eta_{2},1), & (1,1); \\
  ~ & ~ & ~ \\
  (\omega_{1}, \epsilon_{1}), & \cdots (\omega_{l}, \epsilon_{l}), & (\eta_{1}+\eta_{2}+r_{1}+\cdots +r_{n}, 2);
\end{array} q \right],
\end{equation}
where $\Re(\eta_{1})>0,~\Re(\eta_{2})>0$ and max\{$|z_{1}|,\cdots ,|z_{n}|\}<1$.
\end{theorem}

\begin{proof}
We have (see \cite[p.965(20)]{s11})
\begin{equation}\label{4.2} \int_{0}^{1}u^{\eta_{1}-1} (1-u)^{\eta_{2}-1} \prod_{i=1}^{n}(1-z_{i}u)^{-\beta_{i}}du=B(\eta_{1}, \eta_{2})
\end{equation}
\begin{equation*}
\times F_{D}^{(n)}(\eta_{1}, \beta_{1},\cdots \beta_{n}; \eta_{1}+\eta_{2}; z_{1},\cdots z_{n}),
\end{equation*}
where $\Re(\eta_{1})>0,~\Re(\eta_{2})>0$, max\{$|z_{1}|,\cdots|z_{n}|\}<1$ and $F_{D}^{(n)}$ is Lauricella's hypergeometric function of $n$ variables \cite[p.60(4)]{s21}.

On setting $a_{1}=0,~a_{2}=\eta_{3}=1,~h_{2}(u)=u(1-u)$ and $h_{1}(u)=\prod\limits_{i=1}^{n}(1-z_{i}u)^{-\beta_{i}}$ in \eqref{2.1}, expanding the multi-index Mittag-Leffler function $E_{(\epsilon_{i}),(\omega_{i})}$ by its defining series, applying the result \eqref{4.2} on the right side and after some simplification, we obtain our required result \eqref{4.1}.
\end{proof}
Clearly, for $n=2$ Theorem 4.1 immediately reduces to Theorem 2.1.

\begin{theorem}
Let us consider the conditions for $G(u,t)$ defined by \eqref{3.1} to be same as in Theorem 3.1. Then we have
\begin{equation}\label{4.3} \int_{0}^{1} y^{m-1} (1-y)^{n-m-1} \prod_{i=1}^{k}(1-z_{i}y)^{-\beta_{i}}~G\left\{u,t{y}^{\mu}(1-y)^{\nu}\right\} \end{equation}
\begin{equation*}\times E_{(\epsilon_{i}),(\omega_{i})}[qy(1-y)] dy=\sum_{r,r_{1},\cdots,r_{k}=0}^{\infty}{a}_r g_{r}(u){t}^r~\frac{(\beta_{1})_{r_{1}} \cdots (\beta_{k})_{r_{k}}{z_{1}}^{r_{1}}\cdots {z_{k}}^{r_{k}}}{r_{1}!\cdots r_{k}!}
\end{equation*}
\begin{equation*}
\times{}_3\Psi_{l+1}
\left[\begin{array}{ccc}
  (m+\mu r+r_{1}+\cdots+r_{k},1), & (n-m+\nu r,1), \\
  ~ & ~ \\
  (\omega_{1}, \epsilon_{1}), & \cdots (\omega_{l}, \epsilon_{l}),
\end{array}  \right.
\end{equation*}
\begin{equation*}
\left.\begin{array}{cc}
  ~ & (1,1); \\
  ~ & ~ \\
  ~ & (n+\mu r+\nu r+r_{1}+\cdots+r_{k}, 2);
\end{array} q \right]
\end{equation*}
\end{theorem}

\begin{proof}
The proof of this theorem is similar to that of Theorem 3.1. Therefore we omit its proof.\end{proof}

Also, it is noticed that, by using the special cases of the multi-index Mittag-Leffler function $E_{(\epsilon_{i}),(\omega_{i})}$ (given in \eqref{1.6a}, \eqref{1.6}, \eqref{1.7} and \eqref{1.8}) in our main results derived in this paper, we can establish some new families of beta-type integrals involving the Wiman function, Bessel function, Struve function and Lommel function.

\end{document}